\documentclass[12pt]{amsart}
\usepackage{amsmath,amsthm}
\usepackage{amsfonts}
\usepackage[psamsfonts]{amssymb}

\title[Chevalley--Shephard--Todd's theorem]{On Chevalley--Shephard--Todd's theorem in positive characteristic}
\author{Abraham Broer}
\address{D\'epartement de math\'ematiques\\ et de statistique\\
Universit\'e de Montr\'eal\\
C.P. 6128, succursale Centre-ville\\
Montr\'eal (Qu\'ebec)\\
Canada H3C 3J7}

\email{broera@DMS.UMontreal.CA}
\date\today
\thanks{}
\dedicatory{To Gerald Schwarz, on the occasion of his 60-th anniversary}

\theoremstyle{plain}

\newtheorem{theorem}{Theorem}

\theoremstyle{remark}
\newtheorem*{remark}{Remark}

\newtheorem{example}{Example}

\def\F{{\mathbb F}}                            

\def\GL{\mathop{\operatorname{GL}}\nolimits} 
 
\def\GU{\mathop{\operatorname{GU}}\nolimits} 
\def\SU{\mathop{\operatorname{SU}}\nolimits} 
\def\GO{\mathop{\operatorname{GO}}\nolimits} 
\def\SO{\mathop{\operatorname{SO}}\nolimits} 
\def\ROPlus{\mathop{\operatorname{R^+O}}\nolimits} 
\def\Sp{\mathop{\operatorname{Sp}}\nolimits} 
\def\Sg{{\mathfrak S}} 

\def\Tr{\mathop{\operatorname{Tr}}\nolimits}   
\def\Stab{\mathop{\operatorname{Stab}}\nolimits} 

\def\gotq{{\mathfrak q}} 
\def\gotP{{\mathfrak P}} 
\def\gotM{{\mathfrak M}} 

\begin{document} 
\maketitle
\begin{abstract}
Let $G$ be a finite group acting linearly on the vector space $V$ over a field of
arbitrary characteristic. The action is called {\em coregular} if the 
invariant ring is generated by algebraically independent homogeneous invariants and
 the {\em direct summand property} holds if 
there is a  surjective $k[V]^G$-linear map $\pi:k[V]\to k[V]^G$.

The following Chevalley--Shephard--Todd type theorem is proved.
Suppose $V$ is an irreducible $kG$-representation, then the action is coregular if and only if  $G$ is generated by
pseudo-reflections and the direct summand property holds.
\end{abstract}

\section*{Introduction}
Let  $V$ be a  vector space of dimension $n$ over a field $k$. A linear transformation $\tau:V\to V$ is called a {\em pseudo-reflection}, if its fixed-points space $V^\tau=\{v\in V; \tau(v)=v\}$ is a linear subspace
of codimension one. Let $G<\GL(V)$ be a finite group acting linearly on $V$. Then $G$ acts by algebra automorphisms on the coordinate ring $k[V]$, which is by definition
the symmetric algebra on the dual vector space $V^*$. 
We shall say that $G$ is a {\em pseudo-reflection group}
if $G$ is generated by pseudo-reflections; it is called a {\em non-modular } group if $|G|$ is not divisible
by the characteristic of the field. The action is called {\em coregular} if the 
invariant ring is generated by $n$ algebraically independent homogeneous invariants.

A well-known theorem
of Chevalley-Shephard-Todd~\cite[Chapter 6]{Benson} says that if the group is non-modular
then $G$ is a pseudo-reflection group if and only if the action is coregular.

By a theorem of Serre~\cite[Theorem 6.2.2]{Benson} the implication that coregularity of
the action implies that $G$ is a pseudo-reflection group is true even without the condition
of non-modularity. This is not true for the other implication: there are pseudo-reflection groups
whose action is  not coregular.

Coxeter, Shephard and Todd classified all pseudo-reflection groups in characteristic zero.
More recently the irreducible pseudo-reflection groups were classified over
any characteristic, by Kantor, Wagner, Zaleski\u{\i}  and Sere\v{z}kin. Using this
classification Kemper--Malle~\cite{KemperMalle1997} decided which irreducible
pseudo-reflection groups possess the coregular property and which do not.
They observed that those irreducible pseudo-reflection groups that possess the
coregularity property are exactly those such that the actions are coregular for all the point stabilizers of non-trivial subspaces. 

We say that the {\em direct summand property} holds if 
there is a  surjective $k[V]^G$-linear map $\pi:k[V]\to k[V]^G$ respecting the gradings.
For a non-modular group the direct summand property always holds, because in that case
we can take the {\em transfer} $\Tr^G$ as projection, defined by
$$\Tr^G: k[V]\to k[V]^G:\ \Tr^G(f)=\sum_{\sigma\in G}\sigma(f);$$
since for any invariant $f$ we have $\Tr^G({|G|^{-1}f)}=f$.
Also the coregular property implies the direct summand property.

In this article we show first that if the direct summand property holds for $G$ then the
direct summand property holds for all the point-stabilizers of subspaces of $V$,
cf. Theorem~\ref{induction}. Then using this and the results of Kemper--Malle we show
for irreducible $G$-actions that the action is coregular if and only if $G$ is a pseudo-reflection 
group and the direct summand property holds, cf. Theorem~\ref{main}. We conjectured before that this also holds without the irreduciblity condition, cf.~\cite{Broer2005}. Elsewhere we show that the converse is also  true if $G$ is abelian, cf.~\cite{BroerA}.

In the first section we show that the direct summand property is inherited by point-stabilizers.
In the second section we recall Kemper--Malle's classification of irreducible pseudo-reflection
groups that are not coregular, and describe the other tools used in the proof of the main
theorem. In the last section we give the details of the calculations.

\section{The direct summand property and point stabilizers}
For elementary facts on the invariant theory of finite groups we refer to~\cite{Benson},
for a discussion of the direct summand property and the different see~\cite{Broer2005}. 
The transfer map extends to the quotient field of $k[V]$. We recall that the (Dedekind) different $\theta_G$ of the $G$-action on $V$ can be defined as the largest degree homogeneous
form $\theta_G\in k[V]$ such that $\Tr^G(\frac{f}{\theta_G})$ is without denominator, i.e.
$\Tr^G(\frac{f}{\theta_G})\in k[V]^G$, for all $f\in k[V]$; it is unique up to a multiplicative scalar. The direct summand property holds if and only if there exists a $\tilde{\theta}_G\in k[V]$ such
that $\Tr^G(\frac{\tilde{\theta}_G}{\theta_G})=1$ and then we can take as $k[V]^G$-linear projection
$$\pi:k[V]\to k[V]^G:\ \pi(f):=\Tr^G(\frac{\tilde{\theta}_Gf}{\theta_G}).$$

In Kemper--Malle's classification Steinberg's classical result is often used saying that the coregular property is inherited by point stabilizers of linear subspaces.
We shall prove that also the direct summand property is inherited by point stabilizers of linear subspaces. 

The key point in the proof of both results is that the affine group $V^G$ acts on $V$ by translations,  namely
$\tau_u:v\mapsto v+u$ ($u\in V^G$, $v\in V$), commuting with the linear $G$-action.

\begin{theorem}\label{induction}
Let the finite group $G$  act linearly on the vector space $V$ over the field $k$ and
let $H$ be the point-stabilizer of
a linear subspace $U\subset V$. 

If the $G$ action on $V$ has the direct summand property then the 
$H$ action on $V$ also has the direct summand property.
\end{theorem}

\begin{proof}
We write $A:=k[V]$, $C:=k[V]^H$ and $B=A^G$. The prime ideal generated by the linear forms vanishing on $U$ is denoted by $\gotP\subset A$; its intersection with $C$ is the prime ideal $\gotq=\gotP\cap C$. The inertia subgroup of $\gotP$ coincides with $H$:
$$H=\{\sigma\in G; (\sigma -1)(A)\subseteq \gotP\}.$$

Let $\theta_G$ and $\theta_H$ be the two (Dedekind) differents with respect to the $G$-action and the $H$-action on $V$.  In particular
$\Tr^G({A\over \theta_G})\subseteq B$, and $ \Tr^H({A\over \theta_H})\subseteq C.$

Let $V^\alpha\subset V$ be a linear subspace of codimension one, defined as the zero set of the linear form $x_\alpha$. Then $x_\alpha$ divides $\theta_G$ if and only if there is a pseudo-reflection in $G$ with reflecting hyperplane $V^\alpha$, or in other words there exists a $g\in G$ such that for all $a\in A$, $g(a)-a\in x_\alpha A.$
Now for such a pseudo-reflection $g$ we have
$$g\in H\iff V^\alpha\supseteq U\iff x_\alpha\in \gotP.$$
It follows that $\theta_H$ is the part of $\theta_G$ involving the powers of linear forms 
$x_\alpha$, such that $x_\alpha\in \gotP$.  Let $\theta_{G/H}$ be the part of $\theta_G$ involving the powers of linear forms $x_\alpha$, such that $x_\alpha\not\in \gotP$, then $\theta_G=\theta_{G/H}\cdot\theta_H$.
So $\theta_H$ and $\theta_{G/H}$ are relatively prime, and more importantly
$\theta_{G/H}\not\in \gotP$

The homogeneous element $\theta_G$ is a $G$-semi-invariant for some character $\chi:G\to k^\times$. Similarly $\theta_H$ is an $H$-semi-invariant. The quotient $\theta_{G/H}=\theta_G/\theta_H$ is an element of $A$, and is also an $H$-semi-invariant. So there is a power $\theta_{G/H}^e$ that is an absolute $H$-invariant, i.e., $\theta_{G/H}^e\in C$, but 
$$\theta_{G/H}^e\not\in \gotP\cap C=\gotq.$$

Assume now that $B$ is a direct summand of $A$ as graded $B$-module; hence there exists
a homogeneneous $\tilde{\theta}\in A$ such that
$\Tr^G({\tilde{\theta}\over \theta_G})=1$.
We have to prove that the action of $H$ also has the direct summand property, or that the ideal
$$I_H:=\Tr^H({A\over \theta_H})\subseteq C$$
is in fact equal to $C$.

We shall first show that $$\theta_{G/H}^e\in I_H.$$ 
Since $\theta_{G/H}^e\not\in \gotq$,  it will follow that
$I_H\not\subseteq \gotq.$

Let $g_1,\ldots,g_s$ be right coset representatives of $H$ in $G$, i.e., we have a disjoint union $G=\cup_{i=1}^s Hg_i$. Then
\begin{eqnarray*}
\theta_{G/H}^e&=&\theta_{G/H}^e\cdot \Tr^G({\tilde{\theta}\over\theta_G})\\
&=&\theta_{G/H}^e\cdot \Tr^H\left(\sum_{i=1}^sg_i({\tilde{\theta}\over\theta_G})\right)\\
&=&\Tr^H\left(\theta_{G/H}^e\cdot \sum_{i=1}^sg_i({\tilde{\theta}\over\theta_G})\right)\\
&=&\Tr^H\left(\theta_{G/H}^e\cdot {\sum_{i=1}^s\chi^{-1}(g_i)g_i(\tilde{\theta})\over\theta_G})\right)\\
&=&\Tr^H\left(\frac{\theta_{G/H}^{e-1}\cdot \sum_{i=1}^s\chi^{-1}(g_i)g_i(\tilde{\theta})}{\theta_H})\right)\in \Tr^H({A\over\theta_H})=I_H.
\end{eqnarray*}

Suppose now that $I_H$ is a proper ideal.  Since it is a homogeneous ideal of $C$ it is then contained in the maximal homogeneous ideal $\gotM_0$ of $A$, the ideal of polynomials all vanishing at the origin $0\in V$. We shall show that then even $I_H\subseteq \gotP$, which is a contradiction.

To prove this we  can assume that $k$ is algebraically closed.
If $u\in U$, then the affine transformation $\tau_u:v\mapsto v+u$ commutes with the linear $H$ action, since
$$\tau_u(h\cdot v)=h\cdot v+u=h\cdot (v+u)=h\cdot(\tau_u(v)).$$
So it induces an algebra automorphism $\alpha=\tau_u^*$ of $A$ commuting with the $H$-action,
by $$\alpha(f)(v)=(\tau_u^*\cdot f)(v):=f(v-u)$$ moving the maximal ideal $\gotM_0$ into the maximal ideal $\gotM_u$ of 
polynomials in $A$ vanishing at $u$. It also commutes with $\Tr^H$, and fixes the linear forms of $\gotP$, so it fixes $\theta_H$. But then
$$I_H=\Tr^H({A\over \theta_H})=\Tr^H({\alpha(A)\over \theta_H})=
\alpha\left(\Tr^H\left({A\over \theta_H}\right)\right)\subseteq \alpha(\gotM_0)\subseteq \gotM_u.$$
So $I_H\subseteq \cap_{u\in U}\gotM_u$. By Hilbert's Nullstellensatz $\gotP= \cap_{u\in U}\gotM_u$
and $$I_H\subseteq \gotP\cap C=\gotq.$$
This is a contradiction, so $I_H$ is not a proper ideal, i.e.,
$$\Tr^H({A\over\theta_H})=C,$$
which implies that the direct summand property holds for the $H$ action.
\end{proof}

\section{Main result and tools for the proof}
In this section we describe the tools we used to prove our main theorem.

\begin{theorem}\label{main} 
Let $G$ be an irreducible  pseudo-reflection group group acting on $V$. Then the
action is coregular if and only if $G$ is a pseudo-reflection group and the  direct summand
property holds.
\end{theorem}

It is already known that if the action is coregular then $G$ is a pseudo-reflection group and
the direct summand property holds; it follows from Serre's theorem~\cite[Theorem 6.2.2]{Benson}.
For the other direction we use Kemper--Malle's classification of irreducible pseudo-reflection groups not
having the coregular property. We shall use their notation. 

\begin{theorem}[Kemper--Malle~\cite{KemperMalle1997}]\label{KemperMalle}
Let $G$ be an irreducible  pseudo-reflection group group. Then it does not  have the coregular property
if and only if it occurs in the following list. 

(I) (Unitary pseudo-reflection groups) $\SU_n(q)\leq G \leq \GU_n(q)$, $n\geq 4$, and
$\SU_3(q)\leq G < \GU_3(q)$.

(II) (Symplectic pseudo-reflection groups) $\Sp_n(q)$, $n\geq 4$ and $n=2m$ even.

(III-a) (Orthogonal reflection groups of odd characteristic) $q$ odd: $\Omega_n^{(\pm)}(q)< G\leq \GO_n^{(\pm)}(q)$, except
$\GO_3(q)$, $\ROPlus_3(q)$, $\GO_4^{-}(q)$.

(III-b)  (Orthogonal pseudo-reflection groups of even characteristic) $q$ even: $\SO^{(\pm)}_{2m}(q)$, $2m\geq 4$, except $\SO^{-}_4(q)$.

(IV) (Symmetric groups) $\Sg_{n+2}$, $p|(n+2)$, $n\geq 3$.

(V) (Exceptional cases) (i) $W_3(G_{30})=W_3(H_4)$, (ii) $W_3(G_{31})$, (iii) $W_5(G_{32})$, 
(iv) $W_3(G_{36})=W_3(E_7)$, (v) $W_3(G_{37})=W_3(E_8)$, (vi) $W_5(G_{37})=W_5(E_8)$ and
(vii) $W_2(G_{34})=3\cdot U_4(3)\cdot 2_2$.
\end{theorem}

\begin{remark}
In comparing Kemper--Malle's calculations with ours the reader should be aware that
they work with the symmetric algebra of $V$ and we with the coordinate ring of $V$.
See also the comments in \cite{DerksenKemper2000} on Kemper-Malle's article.
\end{remark}

\subsection{Tools}
To prove our theorem we shall exhibit for every pseudo-reflection group in Kemper--Malle's list
an explicit point-stabilizer $H$ such that for the $H$-action on $V$ the direct summand
property does not hold. Then by Theorem~\ref{induction} the $G$-action on $V$ does not have
the direct summand property either.

In most cases we found a point stabilizer $H$ that is a $p$-group. Then we can use the following tools
to show that the $H$-action does not have the direct summand property.

If $H$ is a $p$-group acting on $V$ and the direct summand property holds then $H$ is
generated by its transvections, cf.~\cite[Corollary 4]{Broer2005}. So if $H$ is not generated by transvections then the direct summand property does not hold.

Often $H$ is  abelian. Then we can use that for abelian pseudo-reflection groups the direct summand property holds if and only if the action is coregular, cf.~\cite{BroerA}.

For induction purposes the following trivial remark is useful.
Let $H$ be a group. We shall say that two $kH$-modules $V_1$ and $V_2$ are {\em equivalent} if one is obtained
from the other by adding a trivial direct summand, for example $V_2\simeq V_1\oplus k^r$.
Then the action on one is coregular (has the direct summand property, et cetera) if and only 
if the other is coregular (has the direct summand property, et cetera).

Sometimes the following is useful to disprove coregularity.
If the action is coregular with fundamental degrees $d_1,\ldots,d_n$. Then
$\prod_{i=1}^n d_i=|G|$ and $\sum_{i=1}^n d_i=\delta_G+n$, e.g. \cite[\S 2.5]{Broer2005},
where $\delta_G$ is the differential degree, i.e., the degree of the different $\theta_G$. We give two examples.

\begin{example}\label{abelian}
(i) Let $k=\F_{q^2}$, $q=p^r$ and $V=k^{2n}$, $2n\geq 4$, with standard basis $e_1,\ldots,e_{2n}$.
Consider the group
$$G_n:=\{\left(\begin{matrix}I&O\\B&I\end{matrix}\right);\ \overline{B}=-B^T\},$$
where 
$B$ is an anti-hermitian $n\times n$ matrix with coefficients in $\F_{q^2}$, and where  $\overline{B}_{ij}:=B_{ij}^q$.
It is normalized by 
$$N:=\{\left(\begin{matrix}A&O\\ O&{\overline A}^{-T}\end{matrix}\right);\ A\in \GL(n,\F_{q^2})\},$$
and $N$ acts transitively on the $\frac{q^{2n}-1}{q^2-1}$ hyperplanes containing 
the subspace $<e_{n+1},\ldots,e_{2n}>$.
Take the hyperplane $<e_2,\ldots,e_{2n}>$; its point-stabilizer $H$ consists of all matrices
in $G_n$ where all coefficients
of $B$ are $0$ except possibly $b_{11}$. Its invariant ring has differential degree $q-1$.
So the differential degree of the $G_n$ action is 
$$\delta_{G_n}=\frac{q^{2n}-1}{q^2-1}(q-1)$$

Let $K$ be the point stabilizer of the subspace $<e_{3},\ldots,e_{2n}>$. Then $K\simeq G_2$
and the action of $K$ on $V$ is equivalent to the $G_2$ action on $k^4$.
Suppose  the $G_2$-action is coregular with fundamental degrees $d_1,\ldots,d_4$. Then
$d_1=d_2=1$, since the first two coordinate functions are invariants. And $d_1d_2d_3d_4=|G_2|=q^4$.
So $(d_1,d_2,d_3,d_4)=(1,1,p^r,p^s)$ for some $r\geq 1, s\geq 1$ and $\sum_id_i=\delta_{G_2}+n$, i.e.
$1+1+p^r+p^s=q^3-q^2+q-1+4$.
Implying that $2\equiv 3$ modulo $p$, which is a contradiction. So the action of $K$ 
is not coregular. Since it is abelian it does not have the direct summand property either.

Conclusion: the action of transvection group $G_n$ on $k^{2n}$, $n\geq 2$, is not coregular and does not have the direct summand property.

(ii) Let $k=\F_q$, $q=p^r$, $V=k^{2n}$, $2n\geq 4$, with standard basis $e_1,\ldots,e_{2n}$.
Consider the group
$$G_n:=\{\left(\begin{matrix}I&O\\B&I\end{matrix}\right);\ B=B^T\},$$
where  $B$ is a symmetric $n\times n$ matrix with coefficients in $\F_q$.
It is normalized by 
$$N:=\{\left(\begin{matrix}A&O\\ O&A^{-T}\end{matrix}\right);\ A\in \GL(n,\F_q)\}.$$
As in (i) we can calculate the differential degree of the $G_n$-action, it is 
$$\delta_{G_n}=\frac{q^n-1}{q-1}(q-1)=q^n-1.$$

Then as in (i) we can consider the point-stablizer $H$ of $<e_3,\ldots,e_{2n}>$ and obtain
the same conclusion. The action of $G_n$ on $k^{2n}$, $n\geq 2$ is not coregular and does not have the
direct summand property.
\end{example}

And the last tool we shall use to prove that the direct summand property does not hold is the following.
If the direct summand property holds for the action of $G$ on $V$ and $J\subset k[V]^G$ an ideal,
then $J=(J\cdot k[V])\cap k[V]^G$, cf.\cite[Proposition 6(ii)]{Broer2005}. We give an example of
its use.

\begin{example}\label{GU3}
Let $k=\F_{q^2}$ and $V=k^3$ with standard basis $e_1,e_2,e_3$ and coordinate  functions
$x_1,x_2,x_3$.
Let $H$ be the point stabilizer of $<e_3>$ inside $\GU_3(q)$, and $\tilde{H}$ the
point stabilizer of $<e_3>$ inside $\SU_3(q)$. Or explicitly,
$$H=\{\left(\begin{matrix}1&0&0\\a&b&0\\c&d&1\end{matrix}\right);\ b^{q+1}=1, d=-ba^q,c+c^q+a^{q+1}=0\}$$
and $\tilde{H}$ is the normal subgroup where $b=1$.
Let $\eta$ be a primitive $q+1$-st root of unity and 
$$\tau:=\left(\begin{matrix}1&0&0\\0&\eta^{-1}&0\\0&0&1\end{matrix}\right).$$
Then $H=<\tilde{H},\tau>$.

Both point stabilizer groups have the algebraically independent invariants 
$$x_1,\ F:=x_1x_3^q+x_2^{q+1}+x_3x_1^q\
\hbox{ and } N(x_3):=\prod_{\sigma \in H/\Stab_H(x_3)}\sigma(x_3),$$
of degrees $1, q+1$ and $q^3$ respectively. Since $|H|=(q+1)q^3$ they form a generating set of
the invariant ring $k[V]^H$, cf.~\cite[Proof of Proposition 3.1]{KemperMalle1997}.

Let $$h:=N(x_2)=\prod_{\sigma \in \tilde{H}/\Stab_{\tilde{H}}(x_2)}\sigma(x_2)=x_2\left(x_2^{q^2-1}-x_1^{q^2-1}\right).$$
Then by construction $h$ is $\tilde{H}$-invariant and $\tau\cdot h=\eta h$. So $h^{q+1}$ is the smallest power of $h$ that is $H$-invariant. Let $f$ be any $\tilde{H}$-invariant such that $\tau\cdot f=\eta f$.
Since $\tau$ is a pseudo-reflection we have that $\tau(f)-f=(\eta-1)f$ is divisible by $x_2$.
Since $f$ is also $\tilde{H}$-invariant it is also divisible by every $\sigma(x_2)$, $\sigma\in \tilde H$,
so is divisible by $h$. Using powers $h^i$ we get similar results for other $H$-semi-invariants.
We get
$$k[V]^{\tilde{H}}=\oplus_{i=0}^{q}k[V]^H h^i$$
and so 
$$k[V]^{\tilde{H}}=k[x_1,F,N(x_3),h];$$
in particular it is a hypersurface ring.
Similarly for $H_m=<\tilde{H},\tau^m>$, for $m|(q+1)$, we get
$k[V]^{H_m}=k[x_1,F,N(x_3),h^{(q+1)/m}]$.
In any case 
$$(x_1,F,N(x_3),h^{(q+1)/m})k[V]=(x_1,x_2^{q+1},x_3^{q^3})k[V]=(x_1,F,N(x_3))k[V].$$
If the direct summand property holds for $H_m$ acting on $V$, then for any ideal $J\subset k[V]^{H_m}$ we
have $J=(J\cdot k[V])\cap k[V]^{H_m}$. In particular, it follows for the maximal homogeneous ideal
$k[V]^{H_m}_+$ of $k[V]^{H_m}$ that
$$k[V]^{H_m}_+=(x_1,F,N(x_3),h^{(q+1)/m})k[V]^{H_m}=(x_1,F,N(x_3))k[V]^{H_m}.$$
So $x_1,F$ and $N(x_3)$ generate the maximal homogeneous ideal $k[V]^{H_m}_+$ and also the algebra  $k[V]^{H_m}$.
But this is a contradiction if $m\neq q+1$.
Conclusion: the action of $H_m$ on $V=k^3$ does not satisfy the direct summand property
if $m|(q+1)$ and $m\neq q+1$.
\end{example}

\section{Details}
In this last section we shall establish explicitly for every pseudo-reflection group not having
the coregular property in Kemper--Malle's list in Theorem~\ref{KemperMalle} a point stabilizer not
having the direct summand property.
For more information on some of the involved classical groups, for
example Witt's theorem, see~\cite{Aschbacher}.

\subsection{Families}
(I) (Unitary pseudo-reflection groups) $\SU_n(q)\leq G \leq \GU_n(q)$, $n\geq 4$, and
$\SU_3(q)\leq G < \GU_3(q)$.

Let first $n=2m\geq 4$ be even. Then there is a basis $e_1,\ldots, e_{2m}$ of $V=\F_{q^2}^n$ such
that the associated Gram matrix is 
$$J=\left(\begin{matrix}O&I\\ I&O
\end{matrix}\right),$$
where $I$ is the identity $m\times m$-matrix and $O$ the zero $m\times m$-matrix.
So an $n\times n$ matrix $g$ with coefficients in $\F_{q^2}$ is in $\GU_n(q)$ if 
and only if $g^TJ\overline{g}=J$, where $\overline{g}$ is the matrix obtained from
$g$ by raising all its coefficients to the $q$-th power.
Let $H$ be the point stabilizer in $GU_n(q)$ of the maximal isotropic subspace $U=<e_{m+1},\ldots,e_n>$, so
$$H=\{  \left(\begin{matrix}I&O\\ B&I\end{matrix}\right);\ \overline{B}=-B^T\}.$$
If $\SU_n(q)\leq G \leq \GU_n(q)$ then $H$ is also the point stabilizer in $G$ of $U$,
since the index of $G$ in $\GU_n(q)$ is relatively prime to $p$ and $H$ is a $p$-group.
We encountered this group in example~\ref{abelian}(i), and we conclude that
the direct summand property does not hold for $H$.

If $n=2m+1\geq 5$ is odd, then the stabilizer in $\SU_n(q)\leq G \leq \GU_n(q)$ of a non-singular vector is a reflection group $\SU_{2m}(q)\leq G_1 \leq \GU_{2m}(q)$, so
we can reduce to the even case, which we just handled.

For $\SU_3(q)\leq G < \GU_3(q)$,  see example~\ref{GU3}. This is one of the
rare cases where no point-stablizer could be found that was a $p$-group not
having the direct summand property.

(II) (Symplectic pseudo-reflection groups) $\Sp_n(q)$, $n\geq 4$ and $n=2m$ even.
There is a basis $e_1,\ldots, e_{2m}$ of $\F_q^n$ such that the associated Gram-matrix is
$$J=\left(\begin{matrix}O&I\\ -I&O
\end{matrix}\right),$$
where $I$ is the identity $m\times m$-matrix and $O$ the zero $m\times m$-matrix.
So an $n\times n$ matrix $g$ with coefficients in $\F_q$ is in $\Sp_n(q)$ if 
and only if $g^TJg=J$.
Let $H$ be the point stabilizer of the maximal isotropic subspace $U=<e_{m+1},\ldots,e_n>$, so
$$H=\{  \left(\begin{matrix}I&O\\ B&I\end{matrix}\right);\ B=B^T\}.$$
We encountered this group in example~\ref{abelian}(ii), and conclude that
the direct summand property does not hold for $H$.

(III-a) (Orthogonal reflection groups of odd characteristic) $q$ odd: $\Omega_n^{(\pm)}(q)< G\leq \GO_n^{(\pm)}(q)$, except
$\GO_3(q)$, $\ROPlus_3(q)$, $\GO_4^{-}(q)$.

Let $V=\F_q^n$.
If $n=2m$ is even, then $V$ admits two equivalence classes of non-degenerate
quadratic forms distinguished by their sign; they are not similar. We get
two orthogonal groups $\GO_{2m}^{\pm}(q)$. If $n=2m+1$ is odd then
there are also two equivalence classes of quadratic forms, but they are similar.
For our purposes we need not distinguish the two (classes of) orthogonal groups,
we write $\GO_{2m+1}(q)$. In any case the orthogonal group does not contain transvections and contains two types of reflections (i.e. pseudo-reflections of order two).
If $\sigma$ is a reflection, then its center $(\sigma-1)(V)$ is a one dimensional nonsingular subspace $<u>$. Conversely, to any one dimensional nonsingular subspace
$<u>$ there corresponds a unique reflection. The orthogonal complement
$<u>^\perp$ is an irreducible orthogonal space and there are two possibilities,
so by Witt's lemma there are exactly two conjugacy classes of nonsingular
subspaces $<u>$, hence two conjugacy classes of reflections. Each conjugacy
class generates a normal reflection subgroup of the full orthogonal group of
index $2$. These are the three reflection groups we consider.

Let us first consider $n=2m\geq 4$ and the reflection subgroups $G<\GO^{+}_{2m}(q)$.
So there is a basis $e_1,\ldots, e_{2m}$ of $V$ such
that the associated Gram matrix is 
$$J=\left(\begin{matrix}O&I\\ I&O
\end{matrix}\right),$$
where $I$ is the identity $m\times m$-matrix and $O$ the zero $m\times m$-matrix.
So an $n\times n$ matrix $g$ with coefficients in $\F_{q}$ is in $\GO_n^{+}(q)$ if 
and only if $g^TJ g=J$. Let $H$ be the point stabilizer in $GO_n^{+}(q)$ of the maximal isotropic subspace $U=<e_{m+1},\ldots,e_n>$, then
$$H=\{  \left(\begin{matrix}I&O\\ B&I\end{matrix}\right);\ B=-B^T\}.$$
If $G$ is any of the reflection groups associated to $\GO_{n}^+(q)$ then its index
is $1$ or $2$, so $H$ is also the point stabilizer in $G$ of $U$. Since $H$ is a non-trivial $p$-group and does not contain pseudo-reflections it follows that the direct summand property does not hold for $H$. 

If $n=2m+1\geq 5$ is odd, then there is a non-singular vector $u$ such that
the point  stabilizer of $<u>$  in the reflection group $G<\GO_{2m+1}(q)$ of index $\leq 2$ is a reflection
group $G_1<\GO_{2m}^+$ of index $\leq 2$ acting irreducibly on $u^{\perp}$. We
can use induction.

Consider now $n=2m>4$ and the reflection subgroups $G<\GO^{-}_{2m}(q)$ of index
$\leq 2$. There are two linearly independent non-singular vectors $u_1,u_2$
such that
the point stabilizer of $<u_1,u_2>$ in the reflection group $G<\GO_{2m}^{-}(q)$ of index $\leq 2$ is a reflection
group $G_1<\GO_{2m-2}^+$ of index $\leq 2$ acting irreducibly on $<u_1,u_2>^{\perp}$. We can reduce to the earlier case.

Consider  $\GO_3(q)$, the orthogonal group with respect to the quadratic form $2x_1x_3+x_2^2$.
Let $H$ be the point-stabilizer of $<e_3>$, then
$$H=\{\left(\begin{matrix}1&0&0\\-b&a&0\\\frac{-b^2}{2}&ab&1\end{matrix}\right);\ a^2=1, b\in \F_q\}$$
The point-stabilizer $H^{-}$ of $\GO_3^{-}(q)$ is the subgroup of $H$ where the coefficient $a=1$.
So $H^{-}$ is a $p$-group without transvections, so the direct summand property does not
hold for $H^{-}$.

Let $H$ be the point stabilizer in $\GO_4^{-}(q)$ of an anisotropic line. Then $H$ is isomorphic
to $\GO_3(q)$. So for at least one of the two reflection subgroups of $\GO_4^{-}(q)$ the point
stabilizer $H'$ of the anisotropic line  is $\GO_3^{-}(q)$. So for that one the direct summand
property does not hold. But both reflection subgroups of index two in $\GO_4^{-}(q)$ are
conjugate inside the conformal orthogonal group; thus neither of them has the direct
summand property.

(III-b)  (Orthogonal pseudo-reflection groups of even characteristic) $q$ even: $\SO^{(\pm)}_{2m}(q)$, $2m\geq 4$, except $\SO^{-}_4(q)$.
Let $V=\F_q^n$, where $n=2m\geq 4$ is even. Then $V$ admits two equivalence classes of non-degenerate
quadratic forms distinguished by their sign. We get
two orthogonal groups $\GO_{2m}^{\pm}(q)$. Now the orthogonal groups are
generated by transvections and do not contain reflections.

First consider $n=2m\geq 4$ and a quadratic form with maximal Witt index. Then
there is a basis $e_1,\ldots,e_n$ with dual basis $x_1,\ldots, x_n$ such that the quadratic form becomes
$Q=\sum_{i=1}^m x_ix_{m+i}$ and the Gram matrix
is 
$$J=\left(\begin{matrix}O&I\\ I&O
\end{matrix}\right),$$
where $I$ is the identity $m\times m$-matrix and $O$ the zero $m\times m$-matrix.
So an $n\times n$ matrix $g$ with coefficients in $\F_{q}$ is in $\GO_n^{+}(q)$ if 
and only if $Q=Q\circ g$ (and so $g^TJ g=J$). Let $H$ be the point stabilizer in $GO_n^{+}(q)$ of the maximal isotropic subspace $U=<e_{m+1},\ldots,e_n>$, so
$H$ is the collection of matrices $\left(\begin{matrix}I&O\\ B&I\end{matrix}\right)$
such that $B_{ij}=B_{ji}$ if $1\leq i\neq j\leq m$ and $B_{ii}=0$, for $1\leq i\leq m$.
Since $H$ is a $p$-group without pseudo-reflections, the direct summand property does not hold
for $H$.

Next consider $n=2m\geq 6$ and a quadratic form with non-maximal Witt index.
Then there are two non-singular vectors $u_1,u_2$ such that the
point stabilizer in $\GO_{n}^{-}(q)$ of $<u_1,u_2>$ is $\GO_{n-2}^+(q)$ acting irreducibly
on $<u_1,u_2>^\perp$. And we can reduce to that case.

(IV) (Symmetric groups) $\Sg_{n+2}$, $p|(n+2)$, $n\geq 3$.
Let $W=k^m$ be a vector space over a field of characteristic $p>0$ with basis
$e_1,\ldots,e_m$; we assume $m\geq 5$. The symmetric group $\Sg_m$ acts on $W$ by permuting the basis elements.
The submodule  of codimension one 
$$\tilde{V}=<e_i-e_j; 1\leq i<j\leq m>$$
contains the submodule
spanned by $v=\sum_{i=1}^m e_i$ if and only if $p$ divides $m$. We assume this; so $m=pm'$ for some integer $m'$ and we define
$V$ to be the quotient module $\tilde{V}/<v>$ with dimension $n:=m-2\geq 3$.

For $1\leq j \leq m'$ define $v_j:=\sum_{i=1}^{p} e_{p(j-1)+i}$, then
$v=\sum_{j=1}^{m'} v_j$ and each $v_j\in \tilde{V}$. Write
$\tilde{U_1}=<v_1,\ldots,v_{m'}>\ \subset \tilde{V}$ with image $U_1$ in $V$.

We remark that if for  $\sigma\in \Sg_m$ and $i$ it holds that $\sigma(v_i)\neq v_i$, then since $m\geq 5$ we have
$\sigma(v_i)-v_i\not\in <v>$.  So the point stabilizer of $U_1$ is the natural subgroup
$$H:=\Sg_{\{1,2,\ldots,p\} }\times \Sg_{\{p+1,p+2,\ldots,2p\} }\times\ldots\times \Sg_{\{(m'-1)p+1,\ldots,
m'p\} }\simeq \Sg_p\times \Sg_p\times\ldots \times \Sg_p.$$

Suppose $p$ odd or $p=2$ and $m'$ is even. Then $w:=\sum_{i=1}^m ie_i\in \tilde V$ and we
define $\tilde{U}=\tilde{U_1}+<w>$ with image
$U\subset V$. 
Let $\pi\in H$ such that $\pi(w)=w$ so if $\pi(e_i)=e_j$ then $i$ and $j$ are congruent modulo $p$,
but this is only possible for $\pi\in H$ if $\pi$ is trivial. And if $\pi(w)-w\in\ <v>$, or equivalently if there is a $c\in k$
such that
$$\pi(w)-w=\sum_{i=1}^mie_{\pi(i)}-\sum_{i=1}^m i e_i=
\sum_{i=1}^m(\pi^{-1}(i)-i)e_i=c\sum_{i=1}^m e_i$$
so $\pi^{-1}(i)=i+c$ for all $i$. So $c\in \F_p$ and $\pi$ is a power of
$$\sigma:=(1,2,3,\ldots,p)(p+1,p+2,\ldots,2p)\cdots((m'-1)p+1,(m'-1)p+2,\ldots,m).$$
We conclude that point stabilizer in $G$ of $\tilde{U}$ is now trivial, but the point stabilizer in $G$ of $U$ is not, it is generated by $\sigma$. On the other hand, the fixed point space of $\sigma$ is $U$.
Since the dimension of $U$ is $m'$, its codimension is
$$m-2-m'=(p-1)m'-2>1$$
(if $p=3$ then $m'\geq 2$ and if $p=2$ then $m'\geq 4$, by our assumptions), so $\sigma$ is not
a pseudo-reflection.
So we found a linear subspace whose point stabilizer is a cyclic $p$-group not containing a
pseudo-reflection. So the direct summand property does not hold.

Let now $p=2$ and $m'$ odd and we can assume $k=\F_2$. The point stabilizer $H$ is now an elementary abelian $2$-group of order $2^{m'}$ generated by
the $m'$ transpositions $(1,2), (3,4),\ldots,(m-1,m)$. These are all transvections and the only pseudo-reflections contained in $H$.  We shall show that its invariant ring is not polynomial.
Take as basis $f_1,\ldots, f_{m-2}$ the images in $V$ of the vectors 
$e_1+e_2, e_2+e_3,\ldots, e_{m-2}+e_{m-1}.$ Let $y_1,\ldots,y_m$ be the dual basis.
Then the fixed point set $V^H$ is spanned by $f_1,f_3,f_5,\ldots,f_{m-3}$ and
the fixed-point set $(V^*)^H$ by $y_2,y_4,\ldots,y_{m-2}$. Suppose $k[V]^H$ is a polynomial ring, and
its fundamental degrees $d_1,d_2,\ldots,d_{m-2}$. We must have $|H|=2^{m'}=d_1d_2\ldots d_{m-2}$ and the number of reflections must be $d_1+d_2+\ldots+d_{m'-2}-(m'-2)$. Since we have exactly $m'-1$ independent linear invariants the fundamental invariants degrees there must be $m'-2$ quadratic
generating invariants and one of degree $4$. So the number of reflections is $m'-2+3=m'+1$.
But we have only $m'$ reflections: a contradiction.
So we found a linear subspace whose point stabilizer is an abelian $p$-group, whose ring
of invariants is not a polynomial ring. Therefore the direct summand property does not hold either.

\subsection{Exceptional cases}
Kemper-Malle~\cite{KemperMalle1997} made some explicit calculations to show that several exceptional irreducible reflection groups have a linear subspace whose point stabilizer is not generated by pseudo-reflections or at least its invariant ring is not polynomial. Using MAGMA we checked all these calculations and obtained the more precise result that 
all exceptional irreducible reflection groups without polynomial ring of invariants have in fact a
linear subspace whose point stabilizer is an abelian $p$-group with an invariant ring that is not
a polynomial ring, and so the direct summand property does not hold. In fact in most cases the point stabilizer is not even generated by pseudo-reflections.

(i) $W_3(G_{30})=W_3(H_4)$. According to \cite[p. 76]{KemperMalle1997} there is a point stabilizer of
a two-dimensional linear subspace that is cyclic of order $3$. Since it was already known that  the full pseudo-reflection group has no transvections, it follows that the point stabilizer is not generated by pseudo-reflections. Indeed, we checked that there is a two dimensional linear subspace with point stabilizer of order three and whose generator has two Jordan blocks of size $2$, hence this point stabilizer is an abelian $p$-group not generated by pseudo-reflections.

(ii) $W_3(G_{31})$. According to \cite[p. 76]{KemperMalle1997} there is a point stabilizer of a linear subspace that is not generated by pseudo-reflections and of order $48$, which is not enough for our purposes. We checked that there is a unique orbit of length $960$; fix a point $v$ in
this orbit and let $H$ be its stabilizer (it is indeed of order $48$). Now $H$ has $18$ orbits of length $16$. We took one of them and took the stabilizer, say $K=H_w$. Then it turned out that $K$ has order $3$, generated by a $4\times 4$ matrix whose Jordan form has two blocks of size $2$, so $K$ is the point stabilizer of $<v,w>$ and is a $p$-group not generated by pseudo-reflections.

(iii) $W_5(G_{32})$. According to \cite[p. 79]{KemperMalle1997} there is a one-dimensional linear subspace with point stabilizer a cyclic group of order $5$. Since the full pseudo-reflection group was known to have no transvections it followed that this point stabilizer is not generated by pseudo-reflections. Indeed we checked that there is a unique orbit whose stabilizers have order $5$ and are generated by a $4\times 4$ matrix whose Jordan form has one block of size $4$. So these stabilizers are not generated by pseudo-reflections, not even by $2$-reflections. So by Kemper's theorem \cite[Theorem 3.4.6]{DerksenKemper2000}, the invariant ring is not even Cohen-Macaulay.

(iv) $W_3(G_{36})=W_3(E_7)$. According to \cite[p. 78]{KemperMalle1997} there is a linear subspace
whose point stabilizer of order $24$ is not generated by pseudo-reflections. There is a unique orbit of length $672$; let $N$ be the stabilizer of one of its points , say $v_1$. Now this group $N$ has several orbits of length $180$, but only one of them has stablizers not generated by pseudo-reflections. Take $v_2$ in that orbit and take its stabilizer $N_1=N_{V_1}$ (so its order is $24$ and is not generated by pseudo-reflections). Now this group $N_1$ has orbits whose stabilizers have order $3$. We took $v_3$ in one of those orbits, and
took its stabilizer $N_2$; $N_2$ was cyclic of order $3$, whose Jordan form has two blocks of
size $2$ and one of size three, so $N_2$ is the point stabilizer of $<v_1,v_2,v_3>$ and its generator
is not a pseudo-reflection, not even a $2$-reflection. So the invariant ring is
not even Cohen-Macaulay, \cite[Theorem 3.4.6]{DerksenKemper2000}.

(v) $W_3(G_{37})=W_3(E_8)$. According to \cite[p. 78]{KemperMalle1997} there is a linear subspace having $W_3(E_7)$ as point stabilizer, so by the previous case it also has a linear subspace whose
point stabilizer is a cyclic group of order $3$, whose Jordan form has two blocks of
size $2$ and one each of size one and three. So the invariant ring is not even Cohen-Macaulay, 
\cite[Theorem 3.4.6]{DerksenKemper2000}.

(vi) $W_5(G_{37})=W_5(E_8)$. According to \cite[p. 78]{KemperMalle1997}  there is linear subspace
whose point stabilizer is cyclic of order $5$. Since it was already known that the pseudo-reflection group does not contain any transvections it follows that this point stabilizer is not generated by pseudo-reflections. 
There is a unique orbit whose stabilizers have order $14400$, let $v_1$ be one of its points and
$H$ its stabilizer. Now $H$ has a unique orbit with stabilizers of order $5$, let $v_2$ be one of its
points and $N$ its stabilizer. Then $N$ is indeed cyclic of order $5$ and the Jordan form of its generator
has two blocks of size $4$, so $N$ is the point stabilizer of $<v_1,v_2>$. So the invariant ring is not even Cohen-Macaulay, \cite[Theorem 3.4.6]{DerksenKemper2000}. Larger linear subspaces have a point stabilizer with polynomial ring of invariants.

(vii) $W_2(G_{34})=3\cdot U_4(3)\cdot 2_2$ (it has half the order of $G_{34}$). According to \cite[p. 80]{KemperMalle1997}  there is an explicit three dimensional linear subspace whose point stabilizer $K$ is a $2$-group of order $32$. The point stabilizer is abelian and generated by transvections, but we checked using MAGMA that the 
$K$-invariant  rings of both $V$ and $V^*$ are non-polynomial (compare \cite[p.107]{DerksenKemper2000}).

As shown in  \cite[p.73]{KemperMalle1997} the remaining exceptional cases are all isomorphic as reflection groups to members of one of the families we already considered above.
This finishes the proof of our main theorem.

\end{document}